\documentclass[11pt,twoside]{amsart}
\usepackage{color}
\usepackage{epic}
\usepackage[latin1]{inputenc}
\usepackage[OT1]{fontenc}
\usepackage{amscd}
\usepackage{amsmath}
\usepackage{amsthm}
\usepackage{amssymb}
\usepackage[all]{xy}
\usepackage{xypic}
\xyoption{curve}
\usepackage{ifthen} 
\usepackage{hyperref} 
\usepackage{graphicx}
\usepackage{multirow}

\newtheorem{theorem}{Theorem}[section]

\newtheorem{proposition}[theorem]{Proposition}

\newtheorem{lemma}[theorem]{Lemma}
\newtheorem{corollary}[theorem]{Corollary}

\numberwithin{equation}{section}

\theoremstyle{definition}
\newtheorem{definition}[theorem]{Definition}
\newtheorem{remark}[theorem]{Remark}

\newtheorem{question}[theorem]{Question}

\newcommand{\CC}{\mathbb{C}}

\newcommand{\PP}{\mathbb{P}}
\newcommand{\QQ}{\mathbb{Q}}

\newcommand{\ZZ}{\mathbb{Z}}

\renewcommand{\to}{\xymatrix@1@=15pt{\ar[r]&}}
\renewcommand{\rightarrow}{\xymatrix@1@=15pt{\ar[r]&}}
\renewcommand{\mapsto}{\xymatrix@1@=15pt{\ar@{|->}[r]&}}
\renewcommand{\twoheadrightarrow}{\xymatrix@1@=15pt{\ar@{->>}[r]&}}
\renewcommand{\hookrightarrow}{\xymatrix@1@=15pt{\ar@{^(->}[r]&}}
\newcommand{\congpf}{\xymatrix@1@=15pt{\ar[r]^-\sim&}}

\xyoption{arrow}
\xyoption{curve}
\newdir{ (}{{}*!/-5pt/@^{(}}

\begin{document}

\newboolean{xlabels} 
\newcommand{\xlabel}[1]{ 
                        \label{#1} 
                        \ifthenelse{\boolean{xlabels}} 
                                   {\marginpar[\hfill{\tiny #1}]{{\tiny #1}}} 
                                   {} 
                       } 
\setboolean{xlabels}{false} 

\title[Birationally isotrivial fiber spaces]{Birationally isotrivial fiber spaces}
\author[Bogomolov]{Fedor Bogomolov$^1$}
\address{Feder Bogomolov, Courant Institute of Mathematical Sciences\\
251 Mercer Street\\
New York, NY 10012, USA\\
and 
National Research University Higher School of Economics, Laboratory of Algebraic Geometry, HSE, 7 Vavilova Str., Moscow, Russian Federation, 117312}
\email{bogomolo@cims.nyu.edu}

\author[B\"ohning]{Christian B\"ohning$^2$}
\address{Christian B\"ohning, Fachbereich Mathematik der Universit\"at Hamburg\\
Bundesstra\ss e 55\\
20146 Hamburg, Germany}
\email{christian.boehning@math.uni-hamburg.de}

\author[Graf von Bothmer]{Hans-Christian Graf von Bothmer}
\address{Hans-Christian Graf von Bothmer, Fachbereich Mathematik der Universit\"at Hamburg\\
Bundesstra\ss e 55\\
20146 Hamburg, Germany}
\email{hcvbothmer@gmail.com}

\thanks{$^1$ Supported by NSF grant DMS- 1001662; the financial support from the Government of the Russian Federation within the
framework of the implementation of the 5-100 Programme Roadmap of the National
Research University  Higher School of Economics, AG Laboratory  is acknowledged.
}
\thanks{$^2$ Supported by Heisenberg-Stipendium BO 3699/1-1 and  BO 3699/1-2 of the DFG (German Research Foundation)}

\begin{abstract}
We prove that a family of varieties is birationally isotrivial if all the fibers are birational to each other.
\end{abstract}

\keywords{fiber spaces, birational automorphism groups, Cremona groups, rational varieties}
\subjclass{14E05, 14E07}

\maketitle

\section{Statement of the theorem}\xlabel{sStatement}

Let $k$ be an arbitrary algebraically closed field. We will have to assume that $k$ has infinite transcendence degree over the prime field, hence is uncountable. All varieties, morphisms and rational maps are defined over $k$. We want to prove the following theorem.

\begin{theorem}\xlabel{tMain}
Let $B$ be an irreducible algebraic variety and let 
\[
\xymatrix{
X\ar[d]_f \ar@{^{(}->}[r]^{\iota } &  B \times \mathbb{P}^t\ar[ld]^{\mathrm{pr}_1}\\
B & 
}
\]
be a projective flat family of irreducible varieties $X_b$, $b \in B$. Assume that all fibers $X_b$ ($b$ a closed point of $B$)  are birational to each other. Let $X_0$ be a projective model of the fibers, contained in some $\PP^s$. Then the family is birationally isotrivial, by which we mean that, equivalently, 
\begin{itemize}
\item
there exists a dense open subset $B^0\subset B$, a finite cover $B' \to B^0$ and a commutative diagram
\[
\xymatrix{
X \times_B B' =: X_{B'}\ar[rd]_{f_{B'}}\ar@{-->}[rr]^{\Phi} &  &  B' \times X_0 \ar[ld]^{\mathrm{pr}_1}\\
 & B' &
}
\]
with $\Phi$ birational; 
\item
or, in other words, denoting by $K=k (B)$ the function field of $B$, the \emph{geometric generic fiber} $X\times_B \overline{K}$ is birational, over $\overline{K}$, to $X_0 \times_{k} \overline{K}$. Here $\overline{K}$ is the separable closure.
\end{itemize}
\end{theorem}

\begin{remark}\xlabel{rThmMain}
Note that Theorem \ref{tMain} does not claim that for any given $b \in B$ we can find a $\Phi$ which is defined in a neighborhood of the generic point of the fiber $X_b$. It just says that this is true for a general point of $B$. 
\end{remark}

The projectivity of the family in Theorem \ref{tMain} is not essential:

\begin{corollary}\xlabel{cMain}
Let $g: U \to B$ be a family of algebraic varieties such that all the fibers are birational to each other and $B$ is integral. Then $g$ is birationally isotrivial.
\end{corollary}

\begin{proof}
One can assume $g$ proper, and then use Chow's lemma: $g$ is dominated by a projective morphism $g' : U' \to B$ of $B$-schemes, and the total spaces are isomorphic on open dense subsets. Since $B$ is integral, $g'$ is generically flat and we can use Theorem \ref{tMain}.
\end{proof}

Theorem \ref{tMain} may be seen as the birational analogue of the local triviality theorem of Fischer and Grauert \cite{FG65}, or, in the algebraic case, the statement that a family of projective schemes is locally isotrivial if all fibers are isomorphic (see \cite{Sernesi}, Prop. 2.6.10). Also, it points to another important aspect of the classification of finite subgroups of the Cremona groups as birational monodromy groups of fibrations with rational fibres. 
It is possible that Theorem \ref{tMain} is well known to some experts, but the ones we consulted could not tell us the proof nor say if it was a valid statement, and we could not find a proof in the literature.

We will give two proofs of Theorem \ref{tMain}: the first in Section \ref{sFirstProof} gives additional insight into the structure of certain parameter spaces of rational maps. Moreover, working over $\CC$, it also yields the result under the weaker hypothesis that $B$ is just an analytic space provided one imposes certain extra conditions on the family, e.g. if the loci in $B$ over which $X_b$ remains constant as a subvariety of $\PP^t$, are equidimensional; see Remark \ref{rGeneralization}. 


The second proof, in Section \ref{sSecondProof}, works only for an algebraic base. 

Though the second proof is much shorter, it is less illuminating, less geometric and does not develop structural results about spaces of rational maps as the first one does, which are very useful elsewhere. So we felt that it would be good to include both. We end with some remarks and open problems.

We would like to thank the very helpful referee who suggested a way of removing the dependency on the Eisenbud-Goto conjecture of our first proof of Theorem \ref{tMain}. Lemma \ref{lReferee} and its application in the proof of Proposition \ref{pFamilies} is due to her or him.

\section{First proof}\xlabel{sFirstProof}

We start by proving that certain parameter spaces for birational maps carry natural structures of algebraic varieties. Note that there are some subtleties here.

\begin{remark}\xlabel{rSubtleties}
We recall several facts from \cite{B-F13}. We can introduce the Zariski topology in the set $\mathrm{Bir}_{\le d}(\PP^n )$ of birational self-maps of $\PP^n$ of degree $\le d$ as follows: let 
\[
H_d \subset \PP \left( \left( \mathrm{Sym}^{d} (k^{n+1})^{\vee} \right)^{n+1} \right)
\]
be the locally closed subset of tuples $(h_0, \dots , h_n )$ of homogeneous polynomials $h_i$ of degree $d$ on $\PP^n$ that give a birational self-map. Then there exists a natural surjection $H_d \twoheadrightarrow \mathrm{Bir}_{\le d}(\PP^n )$, and the quotient topology of the Zariski topology on $H_d$ is called the Zariski topology on $\mathrm{Bir}_{\le d}(\PP^n )$. This has good functorial properties, e.g. a set $F \subset \mathrm{Bir}_{\le d}(\PP^n )$ is closed if and only if for any family of birational self-maps of $\PP^n$ parametrized by a variety $A$, the preimage of $F$ under the induced map $A \to \mathrm{Bir}_{\le d}(\PP^n )$ is closed. But, due to the funny ways birational maps can degenerate to ones of lower degree, $\mathrm{Bir}_{\le d}(\PP^n )$ cannot even be homeomorphic to an algebraic variety; more precisely, every point $p\in \mathrm{Bir}_{\le d}(\PP^n )$ is an \emph{attractive point} for a suitable closed irreducible subset $T \ni p$, by which we mean that $p$ is contained in every infinite closed subset $F \subset T$, and $T$ is such that it contains infinite proper closed subsets. 

The following example of \S 3 of \cite{B-F13} is useful to keep in mind (it is the reason for all these phenomena): define a $2$-parameter family in $\mathrm{Bir}_{\le 2}(\PP^n)$ by the formula
\[
\left( X_0 (aX_2 + cX_0) : X_1 (aX_2 + b X_0) : X_2 (aX_2 + cX_0) : \dots : X_n (aX_2 + cX_0) \right)
\]
or, in affine coordinates $x_i = X_i/X_0$
\[
(x_1, \dots , x_n) \mapsto \left( x_1 \cdot \frac{ax_2 + b}{ax_2 + c}, x_2, \dots , x_n \right) 
\]
where $(a:b:c) \in \PP^2 \backslash \{ (0:1:0), (0:0:1) \} =: \hat{V}$. The image $V$ of $\hat{V}$ in $\mathrm{Bir}_{\le 2}(\PP^n)$ is closed, but the line $L\subset \hat{V}$ given by $b=c$ is contracted to the identity. The topology on $V$ is the quotient topology of the Zariski topology on $\hat{V}$, so the identity is attractive for $V$.
\end{remark}

\begin{definition}\xlabel{dSpacesBirMaps}
Let $S \subset \PP^s = \PP (k^{s+1})$ be an irreducible projective variety (the ``source"). Let $x_0, \dots , x_s$ be homogeneous coordinates in $\PP^s$. We will always assume that $S$ is not contained in any coordinate hyperplane. Let $\PP^t$ be another projective space (the ``target") with homogeneous coordinates $y_0, \dots , y_t$.
\begin{itemize}
\item[(1)]
We write
\[
\mathcal{P}_d (\underline{x}) = \PP \left( \left( \mathrm{Sym}^d (k^{s+1})^{\vee} \right)^{t+1} \right)
\] 
for the projective space of all $(t+1)$-tuples $\underline{p} = (p_0, \dots , p_t)$ of homogeneous polynomials $p_i$ of degree $d$ in the $x_0, \dots , x_s$. Similarly, we denote by
\[
\mathcal{P}_{d'} (\underline{y}) = \PP \left( \left( \mathrm{Sym}^{d'} (k^{t+1})^{\vee} \right)^{s+1} \right)
\]
the set of $(s+1)$-tuples $\underline{q} = (q_0, \dots , q_s)$ of homogeneous polynomials $q_j$ of degree $d'$ in the $y_0, \dots , y_t$.
\item[(2)]
Define the subset
\[
\mathrm{Rat}_d (S, \PP^t) \subset \mathcal{P}_d (\underline{x})
\]
to be the set of those $\underline{p}$ which have the following properties:
\begin{itemize}
\item[(a)]
not all the $p_i$ are simultaneously contained in the homogeneous ideal $I(S) \subset k [x_0, \dots , x_s]$.
\item[(b)]
the rational map (well-defined by (a))
\begin{align*}
\varphi_{\underline{p}} :& S \dasharrow \PP^t\\
   & (x_0: \dots : x_s ) \mapsto (p_0 (x_0, \dots , x_s) : \dots : p_t (x_0, \dots , x_s))
\end{align*}
maps $S$ dominantly to an image $T:= \overline{\varphi_{\underline{p}} (S)} \subset \PP^t$ of dimension $\dim T = \dim S$. 
\end{itemize}
\item[(3)]
Fix moreover a positive integer $d'$. We denote by 
\[
\mathrm{Bir}_{d,d'} (S, \PP^t ) \subset \mathrm{Rat}_d (S, \PP^t) \subset \mathcal{P}_d (\underline{x})
\]
the subset of those $\underline{p}$ which have the following additional property:
\begin{itemize}
\item[(c)]
there exists $\underline{q} \in \mathcal{P}_{d'} (\underline{y})$ such that 
\begin{align*}
\varphi_{\underline{q}} :& T \dasharrow \PP^s\\
   & (y_0: \dots : y_t ) \mapsto (q_0 (y_0, \dots , y_t) : \dots : q_s (y_0, \dots , y_t))
\end{align*}
is a well-defined rational map on $T$ and inverse to $\varphi_{\underline{p}}$. 
\end{itemize}
In other words, $\varphi_{\underline{p}}$ is birational with inverse given by degree $d'$ polynomials.
Let $\mathrm{Bir}_{d} (S, \PP^t )$ be the union, over $d'$, of all $\mathrm{Bir}_{d,d'} (S, \PP^t )$. 
\end{itemize}
\end{definition} 

The following Lemma, its proof and its application in the proof of Proposition \ref{pFamilies} were kindly suggested by the referee. Our initial proof was conditional on the Eisenbud-Goto conjecture.

\begin{lemma}\xlabel{lReferee}
Let $p\colon Z \to \Omega$ be a morphism of varieties over an algebraically closed field $k$ 
and let $W \subset Z$ be a constructible subset. Then there is a finite collection of locally closed subsets $\Omega_i \subset \Omega$, $1\le i\le N$, and closed subsets $Y_i \subset p^{-1}(\Omega_i)$ such that $\Omega= \bigcup_{i=1}^N \Omega_i$ and for any $i$, $1\le i\le N$, and any point $\omega \in \Omega_i$, the fiber $(Y_i)_{\omega}$ is equal to the closure of $W_{\omega}$ in $Z_{\omega}$. 
\end{lemma}

\begin{proof} Using the fact that taking the closure of sets commutes with finite unions, one reduces the Lemma to the case when $Z$ and $\Omega$ are irreducible, $p$ is dominant, and $W$ is an open dense subset in $Z$. Moreover, by Noetherian induction, it is enough to show that there is a non-empty open subset $U \subset \Omega$ such that for any point $\omega \in U$, the fiber $W_{\omega}$ is dense in $Z_{\omega}$.

First assume that the geometric generic fiber of $p$ is irreducible. Then there is a non-empty open subset $U \subset \Omega$ such that for any point $\omega \in U$, both fibers $Z_{\omega}$ and $W_{\omega}$ are irreducible and have the same dimension $\dim(Z) - \dim(\Omega )$. Hence $W_{\omega}$ is dense in $Z_{\omega}$ (actually, this is the only case that appears in the application below).

Secondly assume that $p$ is finite (and surjective). Then the closed complement $C := Z \backslash W$ does not dominate $\Omega$, otherwise $Z$ would not be irreducible. Therefore we can put $U := \Omega \backslash p(C)$. 

Finally, for an arbitrary $p$ as above, changing $\Omega$ by a non-empty open subset, we can decompose $p$ into a composition $Z \xrightarrow{q} \Omega' \xrightarrow{r} \Omega$, where $q$ has an irreducible geometric generic fiber and $r$ is finite. With this aim, consider a separable closure $K$ of $k(\Omega)$ in $k(Z)$. Since $K$ is finitely generated over $k(\Omega)$, changing $\Omega$ by a non-empty open subset, we can achieve that $K$ is generated by functions that are regular on $Z$. The sheaf of $\mathcal{O}_{\Omega}$-algebras generated by them gives us $\Omega'$.
\end{proof}

\begin{proposition}\xlabel{pFamilies}
There are countably many locally closed subvarieties \[ \mathrm{Rat}_{d,i}(S, \PP^t) \subset \mathcal{P}_d(\underline{x}), i\in I,\] 
and flat projective families 
\[
\mathcal{T}_{d,i}(S, \PP^t)\subset \PP^t \times  \mathrm{Rat}_{d,i}(S, \PP^t) \to \mathrm{Rat}_{d,i}(S, \PP^t)
\]
with
\[
\mathrm{Rat}_d(S, \PP^t) = \bigcup_{i\in I} \mathrm{Rat}_{d,i}(S, \PP^t)
\]
such that the closure of the fiber $(\mathcal{T}_{d,i}(S, \PP^t))_{\underline{p}}$ over $\underline{p} \in \mathrm{Rat}_d(S, \PP^t)$ is equal to the image $T$ of $\varphi_{\underline{p}}$.
\end{proposition}

\begin{proof}
The condition that not all of the $p_i$ in $\underline{p}$ are contained in $I(S)$ defines an open subset $\Omega_1 \subset \mathcal{P}_d (\underline{x})$. Moreover, we can ensure a priori that the dimension of the image of $\varphi_{\underline{p}}$ is equal to $\dim S$ by requiring the (equivalent to $\dim S = \dim T$) condition that the Jacobian matrix of the $p_i$'s is of maximal rank generically on $S$, so that $S \dasharrow T$ is generically a covering map. This defines another open set $\Omega \subset \Omega_1$. 

Now apply Lemma \ref{lReferee} to this $\Omega$, $Z = \PP^t \times \Omega$, and $W$ equal to the image of the natural morphism representing the rational $S\times \Omega \dasharrow \PP^t \times \Omega$ on its domain of definition. Applying the Lemma and taking a flattening stratification of all arising morphisms to locally closed subsets of $\Omega$, we obtain locally closed subsets $\mathrm{Rat}_{d,i}(S, \PP^t)$ and families $\mathcal{T}_{d,i}(S, \PP^t) \to \mathrm{Rat}_{d,i}(S, \PP^t)$ as required. 
\end{proof}

We recall the following result of Mumford \cite{Mum66}, Chapter 14.

\begin{theorem}\xlabel{tMumford}
Fix $h = h(n) \in \mathbb{Q}[n]$ with $h (\ZZ ) \subset \ZZ$. Then there exists a uniform $d_h$ such that for each variety $T \subset \PP^t$ with Hilbert polynomial $h$, the Hilbert function of $T$ agrees with the Hilbert polynomial of $T$ at places  $d \ge d_h$ and $I(T)$ is generated by polynomials of degree $\le d_h$. 
\end{theorem}

\begin{proposition}\xlabel{pBirFamilies}
The subset \[ \mathrm{Bir}_{d,d',i}(S, \PP^t):= \mathrm{Bir}_{d,d'}(S, \PP^t) \cap  \mathrm{Rat}_{d,i}(S, \PP^t)\] is closed in $ \mathrm{Rat}_{d,i}(S, \PP^t)$, provided $d' =d'(i)\ge d_h$ for the common Hilbert polynomial $h$ of the fibers. Hence there are countably many flat projective families 
\[
\mathcal{T}_{d,d',i}(S, \PP^t)\subset \PP^t \times  \mathrm{Bir}_{d,d',i}(S, \PP^t) \to \mathrm{Bir}_{d,d',i}(S, \PP^t)
\]
with
\[
\mathrm{Bir}_{d}(S, \PP^t) = \bigcup_{i\in I} \mathrm{Bir}_{d,d'(i),i}(S, \PP^t)
\]
such that the closure of the fiber $(\mathcal{T}_{d,d',i}(S, \PP^t))_{\underline{p}}$ over $\underline{p} \in \mathrm{Bir}_{d,d'}(S, \PP^t)$ is equal to the image $T$ of $\varphi_{\underline{p}}$.
\end{proposition}

\begin{proof}
Consider the set
\begin{gather*}
\mathcal{P}_{d, d',i}:= \\
\left\{ (\underline{p}, [\underline{q}] ) \mid \underline{p}\in \mathrm{Rat}_{d,i} (S, \PP^t) , [\underline{q}] \in \PP ((\mathrm{Sym}^{d'} (k^{t+1})^{\vee} / I (T)_{d'})^{s+1} ), \; T =  \overline{\varphi_{\underline{p}} (S)} \right\}
\end{gather*}
which is a projective bundle over $\mathrm{Rat}_{d,i} (S, \PP^t)$ with fiber over a point $\underline{p}$ the projective space of $(s+1)$-tuples of homogeneous polynomials of degree $d'$ modulo those vanishing on the image of the rational map induced by $\underline{p}$. Here we use the constancy of $\dim I(T)_{d'}$, i.e. we use the condition $d' \ge d_h$. We define a closed subvariety $\mathcal{S}_{d, d',i}$ of $\mathcal{P}_{d, d',i}$ by the requirement that for the pair  $(\underline{p}, \underline{q})$ the two by two minors of the matrix 
\begin{gather*}
\left(
\begin{array}{ccc}
q_0 (p_0 (\underline{x}), \dots , p_t (\underline{x})) & \dots & q_s (p_0 (\underline{x}), \dots , p_t(\underline{x})) \\
x_0 & \dots & x_s
\end{array}
\right)
\end{gather*}
are all zero modulo the ideal $I(S)$. Clearly this condition is independent of the lift $\underline{q}$ of $[\underline{q}]$ to $\mathcal{P}_{d'} (\underline{y})$. If all the polynomials in the first row of the preceding matrix are nonzero modulo $I(S)$, then $\underline{p}$ defines a birational map from $S$ unto its image with inverse induced by $\underline{q}$. 

In the opposite case, $\varphi_{\underline{p}}$ would contract $S$ into a proper subvariety of $T$, the one defined by $\{ q_j = 0 \}$. Here we are using the assumption that $S$ is contained in no coordinate hyperplane $x_j =0$ and that the ideal $I(S)$ is prime. Now contraction is impossible because the dimension of the image of $S$ is equal to the dimension of $S$. 

The projection \[ \mathcal{S}_{d, d',i} \to \mathrm{Rat}_{d,i} (S, \PP^t)\]  is proper because the projection \[ \mathcal{P}_{d, d',i} \to \mathrm{Rat}_{d,i} (S, \PP^t) \] is proper. By construction, the set $\mathrm{Bir}_{d,d',i} (S, \PP^t )$ is equal to the image, a closed subvariety of $\mathrm{Rat}_{d,i} (S, \PP^t)$.
\end{proof}

\begin{proposition}\xlabel{pCover}
Let $\mathrm{Hilb}_{h, \PP^t}^0$ be the open subset of the Hilbert scheme of subschemes $Z$ of $\PP^t$ with Hilbert polynomial $h$ which are reduced and irreducible (hence varieties). Let 
\[
(\mathcal{B}ir-X_0)_{h, \PP^t} \subset \mathrm{Hilb}_{h, \PP^t}^0
\]
be the subset consisting of those $Z$ which are birational to some fixed model variety $X_0 \subset \PP^s$. Then there are countably many locally closed subvarieties $\mathcal{H}_i \subset \mathrm{Hilb}_{h, \PP^t}^0$ such that
\[
(\mathcal{B}ir-X_0)_{h, \PP^t} = \bigcup_i \mathcal{H}_i .
\]
\end{proposition}

\begin{proof}
This follows immediately from Proposition \ref{pBirFamilies}.
\end{proof}

\begin{proof}(of Theorem \ref{tMain})
We get a morphism
\[
\alpha : B \to \mathrm{Hilb}_{h, \PP^t}^0
\]
and its image $\alpha (B)$ contains an open subset $U$ of its closure $\overline{\alpha (B)}$ in $\mathrm{Hilb}_{h, \PP^t}^0$. Then $U$ is covered by the countable union of the locally closed subvarieties $\mathcal{H}_i$. The intersection $U \cap \mathcal{H}_i$ is either contained in a proper subvariety or contains a Zariski dense open subset of $U$. Hence there is some $\mathcal{H}_{i_0}$ containing an open dense subset $U'$ of $U$. This is the point where we use that $k$ is uncountable!  We also have a natural morphism
\[
\beta : \mathrm{Bir}_{d, d',i_0} (X_0, \PP^t) \to \mathrm{Hilb}_{h, \PP^t}^0
\]
such that the closure of the image of $\beta$ contains the closure of $\mathcal{H}_{i_0}$ as an irreducible component. Let 
\[
B^{\sharp} =\alpha^{-1} (U') \subset B, \; \beta^{-1} (U') \subset \mathrm{Bir}_{d, d',i_0} (X_0, \PP^t)
\]
be the indicated open preimages. Shrinking $U'$ a little more and passing to a subvariety, we can assume that there is a $\tilde{B} \subset \beta^{-1} (U')$ such that $\beta$ maps $\tilde{B}$ unto $U'$ and is finite. We get a commutative diagram
\[
\xymatrix{
B' :=B^{\sharp} \times_{U'} \tilde{B} \ar[r]^{\alpha'}\ar[d]^{\beta'} & \tilde{B} \ar[d]^{\beta}\\
B^{\sharp} \ar[r]^{\alpha} & U' \subset  \mathrm{Hilb}_{h, \PP^t}^0
}
\]
Here $\beta'$ is finite, and the pull-back of the universal family on $\mathrm{Hilb}^0_{h, \PP^t}$ under $\beta$ to $\tilde{B}$ is birationally trivial: this is so because by construction the family over $\mathrm{Bir}_{d, d',i_0} (X_0, \PP^t)$ is birationally trivial. Pulling this family back via $\alpha'$ to $B'$ gives the same thing as pulling back the restriction of the original family $f : X_{B^{\sharp}} \to B^{\sharp}$ to $B'$: hence $\beta' : B' \to B$ has the properties claimed in Theorem \ref{tMain}.
\end{proof}

\begin{remark}\xlabel{rGeneralization}
Working over $\CC$, the preceding proof goes through without change if the family $f: X\to B$ is only over an analytic base $B$, provided one imposes the condition that the image $\alpha : B \to \mathrm{Hilb}_{h, \PP^t}^0$ is analytic (it is always constructible in the algebraic case, but may be wilder in the analytic case). For example, it will again be an analytic set if the fibers $\alpha^{-1}(\alpha (b))$ over points in the image all have the same dimension, and in several other cases, see the introduction in \cite{Huck71}.
\end{remark}

\section{Second proof}\xlabel{sSecondProof}

We continue assuming that $k$ has infinite transcendence degree over the prime field $\mathbb{K}$ (and is algebraically closed). The diagram $f : X \to B$ (i.e. all of $X$, $B$, $f$) are defined over a finitely generated extension $k_0$ of $\mathbb{K}$ which is contained in $k$. Then $B$ has a $k_0$-generic point in the sense of van der Waerden, i.e. a closed point $\xi \in B$ with coordinates in $k$ such that any polynomial with coefficients in $k_0$ which vanishes at $\xi$ vanishes identically on $B$, and such that the residue field $\kappa_{\xi}$ is the field $k_0 (B)$ of rational functions with coefficients in $k_0$ on $B$. 

Now $X_{\xi}$ is, by hypothesis, birational to $X_0$ over $k$, which means that, possibly after enlarging $k_0$ by adjoining finitely many elements from $k$ which are algebraically independent over $k_0 (B)$, $X_{\xi}$ is birational to $X_0$ over $\overline{k_0 (B)}$. But then the geometric generic fiber of the family $f : X \to B$ is also birational, over $\overline{k(B)}$, to $X_0$.

\section{Open questions}\xlabel{sQuestions}
\begin{question}\xlabel{qDiscussion}
We finally want to expand on Remark \ref{rThmMain}. Note that Theorem \ref{tMain} says that a family $X \to B$ all of whose fibers are birational is birationally isotrivial, but the trivialization $\Phi$ need not be defined in the generic point of a fiber $X_b$ for a particular $b \in B$. In that respect, it is not the exact birational analogue of the local triviality theorem of Fischer-Grauert  \cite{FG65} (in the holomorphic setting) asserting that a family of compact complex manifolds is locally analytically trivial if and only if all the fibers are biholomorphic: here the family is locally trivial around \emph{every point} of the base. In the algebraic case, the statement is that a family of projective schemes is locally isotrivial if all fibers are isomorphic (see \cite{Sernesi}, Prop. 2.6.10).

On the other hand, if one drops the compactness hypothesis, the Fischer-Grauert theorem becomes false: for example, if one takes 
\[
\mathcal{C} : = \PP^1 \times \PP^1 - \left( \{ (x,x) | x\in \PP^1 \} \cup \{ (x,y) \in \PP^1\times \PP^1 \mid y = a \} \cup \{ (a, b\} \right)
\]
for $a\neq b$ in $\PP^1$ fixed, with projection $\mathrm{pr}_1 : \mathcal{C} \to \PP^1$, then every fiber is $\CC^{\ast} \simeq \PP^1 -$(two points), but $\mathcal{C}$ is not even topologically locally trivial around $a\in \PP^1$: a small circle in $\mathcal{C}_a$ around $(a,b)$ is homologically trivial in $\mathrm{pr}_1^{-1}(U)$ for any neighborhood $U\ni a$ in $\PP^1$.

In this last example, the family is locally holomorphically trivial around a general point of the base, but not even topologically locally trivial in points of a proper analytic subset. It is thus an interesting open question how the situation is in the birational set-up: given a family $f : X \to B$ with all fibers birational to each other, could it happen that for some bad points $b\in B$ there is no birational trivialization $\Phi$ defined in the generic point of $X_b$ (analogue of the last example with open fibers), or does such a $\Phi$ always exist (birational analogue of Fischer-Grauert)? 
\end{question}

\begin{question}\xlabel{qCountable}
It seems an interesting question to investigate if Theorem \ref{tMain} remains true if we only assume $k$ algebraically closed. We may also ask: given a family $f : X \to B$ over $\CC$ which is defined over $\bar{\QQ}$ and such that all fibers over algebraic points in $B$, i.e. $\bar{\QQ}$-valued points, are birational to a fixed model, is it true that the fibers in a Zariski dense open subset of $B(\CC )$ are birational to each other?
\end{question}

\begin{remark}\xlabel{rStronglyRational}
Concerning Question \ref{qCountable}, we would like to remark that there are families of unirational varieties where the birational type is expected to change on a countable union of subvarieties of the base: e.g. this is expected to happen for the family of cubic fourfolds where the very general one should be irrational whereas in a countable union of subvarieties of the parameter space one can get rational ones. But possibly these are not dense in a way that would yield a negative answer to Question \ref{qCountable}.

Let us also remark that if one considers \emph{strongly rational varieties}, i.e. smooth varieties $X$ which contain an open subset $U$ isomorphic to an open subset $V\subset \PP^n$ such that $\PP^n -V$ has codimension at least $2$ in $\PP^n$, then their deformation properties are much better: smooth small deformations of strongly rational varieties are again strongly rational, \cite{IN03} Thm. 4.5.
\end{remark}


\begin{thebibliography}{9999999999}

\bibitem[BCW82]{BCW82}
Bass, H., Connell, E.H., Wright, D., \emph{The Jacobian conjecture: reduction of degree and formal expansion of the inverse}, Bulletin (New Series) of the American Mathematical Society, Vol. \textbf{7}, Number 2 (1982), 287--330

\bibitem[BayMum93]{BayMum93}
Bayer., D. \& Mumford, D., \emph{What can be computed in algebraic geometry?}, in Computational Algebraic Geometry and Commutative Algebra (Cortona 1991), Sympos. Math. XXXIV, Cambridge University Press, Cambridge U.K. (1993), 1--48

\bibitem[B-F13]{B-F13}
Blanc, J. \& Furter, J.-P., 
\emph{Topologies and structures of the Cremona groups},  Ann. of Math. \textbf{178} (2013), no. 3, 1173--1198

\bibitem[E-G84]{E-G84}
Eisenbud, D. \& Goto, S., \emph{Linear free resolutions and minimal multiplicity}, J. Algebra \textbf{88}, (1984), 89--133 

\bibitem[FG65]{FG65}
Fischer, W. \& Grauert, H., \emph{Lokal-triviale Familien kompakter komplexer Mannigfaltigkeiten}, Nachrichten der Akademie der Wissenschaften in G\"ottingen (1965) Nr. \textbf{6}, G\"ottingen: Vandenhoeck \& Ruprecht, 90--94

\bibitem[Huck71]{Huck71}
Huckleberry, A. T.,  \emph{On local images of holomorphic mappings}, Annali della Scuola Normale Superiore di Pisa - Classe di Scienze vol. \textbf{25} issue 3 (1971), 447--467

\bibitem[IN03]{IN03}
Ionescu, P. \& Naie, D., \emph{Rationality properties of manifolds containing quasi-lines}, 
International Journal of Mathematics vol. \textbf{14}, No. 10 (2003), 1053--1080

\bibitem[Mum66]{Mum66}
Mumford, D., \emph{Lectures on curves on an algebraic surface}, with a section by G. M. Bergman, Annals of Mathematics Studies, No. \textbf{59}, Princeton University Press, Princeton, N.J., (1966)

\bibitem[Sernesi]{Sernesi}
Sernesi, E., \emph{Deformations of algebraic schemes}, Grundlehren der Mathematischen Wissenschaften Volume \textbf{334}, Springer (2010)
\end{thebibliography}
\end{document}